\documentclass[reqno,11pt]{amsart}

\usepackage{amsfonts, amssymb, verbatim}
\usepackage{amssymb,epsfig,amscd, verbatim}
\usepackage{amsthm}
\usepackage{amsmath}
\usepackage{accents}
\usepackage{graphicx}
\usepackage{latexsym}

\usepackage[all]{xy}

\usepackage{array}
\usepackage{float}
\usepackage{wrapfig}

\newtheorem{theorem}{Theorem}

\newtheorem{proposition}[theorem]{Proposition}

\theoremstyle{definition}

\numberwithin{equation}{section} \numberwithin{theorem}{section}

\theoremstyle{remark}
\newtheorem{remark}[theorem]{Remark}

\newtheorem{claim}[theorem]{Claim}

\def\z{\,_{\dot z}\,}

\def\vac{|0\rangle}                            

\newcommand\lbb[1]{\label{#1}}

\def\<{\langle}
\def\>{\rangle}

\def\vac{\mathbf{1}}                            

\def\al{\alpha}                         

\def\z{\,_{\dot z}\,}

\def\noi{\noindent}

\def\v2a{(V,\z,\vac,d)}
\begin{document}

\title[Pell equation:
Generalizations of continued fraction and Chakravala
 algorithms]{Pell equation:
A generalization of continued fraction and Chakravala
 algorithms using the LLL-algorithm\ \
}

\author[Jos\'e I. Liberati]{Jos\'e I. Liberati$^*$}
\thanks {\textit{$^{*}$Ciem - CONICET, Medina Allende y
Haya de la Torre, Ciudad Universitaria, (5000) C\'ordoba -
Argentina. \hfill \break \indent \hskip .2cm e-mail:
joseliberati@gmail.com
\hfill \break \indent
\hfill \break \indent
ORCID Number: 0000-0002-5422-4056}}
\address{{\textit{Ciem - CONICET, Medina Allende y
Haya de la Torre, Ciudad Universitaria, (5000) C\'ordoba -
Argentina. \hfill \break \indent e-mail: joseliberati@gmail.com}}}

\date{ 16 jul, 2024}


\dedicatory{Dedicated
 to my parents, H\'ector and Sheila, with  heartfelt gratitude}

\maketitle

\begin{abstract}
We introduce a generalization of the continued fraction and Chakravala algorithms for solving the Pell equation, utilizing  the LLL-algorithm for rank 2 lattices.
\end{abstract}


\

\section{Introduction}\lbb{intro}

\vskip .3cm

In 628, Brahmagupta was the first to discover the identity
(\ref{3.2}), which states that if the triples $(a,b,k)$ and $(m,l,s)$
satisfy the equations
\begin{equation*}
    a^2-d b^2=k \qquad \hbox{ and }\qquad m^2-d l^2=s,
\end{equation*}
then
\begin{equation}\label{3.2}
   (a^2-d b^2)(m^2-d\, l^2)= (am+db\,l)^2-d(al+b\,m)^2=ks.
\end{equation}

\vskip .2cm

\noindent This allows the {\it composition} of two solution
triples $(a,b,k)$ and $(m,l,s)$ into a new triple
\begin{equation}\label{3.3}
    (am+db_{_{\,}}l,al+b\,m,ks).
\end{equation}

In the Chakravala method, discovered by Bhaskara II in the 12th century, the core idea is that given a triple $(a, b, k)$ (which satisfies $a^2 - db^{, 2} = k$), we can compose it with the trivial triple $(m, 1, m^2 - d)$ (by setting $l=1$ in (\ref{3.3})) to obtain a new triple $(am + db, a + b,m, k(m^2 - d))$ which can be scaled down by $k$ to yield
\begin{equation}\label{3.4}
\bigg(\frac{am + db}{k}\bigg)^2 - d\bigg(\frac{a + b,m}{k}\bigg)^2 = \frac{m^2 - d}{k}
\end{equation}
\vskip .1cm

\noindent and then we choose $m$ to be a positive integer for which the three quotients in (\ref{3.4}) are integers and minimize the absolute value of $m^2 - d$. This results in a new triple of integers and the process is continued until a stage is reached at which the equation has the desired form
\begin{equation*}\label{pell}
a^2 - db^2 = 1,
\end{equation*}
solving the Pell equation.

The fundamental concept of our algorithm is straightforward: we can replicate the entire argument with the inclusion of "$l$" (instead of setting $l=1$), and in each step, we need to choose the two variables $m$ and $l$ that satisfy specific conditions. One of the principal outcomes of this study consists of two variations of this generalization, each imposing different conditions. These are defined as the \textit{First and Second Algorithms with L} in Section 4.

The other primary results of this research involve the implementations and enhancements of the Second Algorithm with L presented in Section 5. Notably, if we assume knowledge of a close lower bound of the regulator (not necessarily its integer part), we have developed an algorithm employing the LLL-algorithm with rank 2 lattices to identify the fundamental solution of the Pell equation. This is referred to as the \textit{Second Algorithm with LLL}, as detailed in Subsection \ref{lL}. Here, the LLL-algorithm refers to the algorithm as defined in \cite{LLL}. Utilizing analogous concepts, we introduce a generalization of the continued fraction algorithm in Section \ref{7ma}.

We identify the study of convergence and  its computational complexity  as open problems. We believe  that the concepts introduced in this work warrant further exploration and detailed examination.

Sections 2 and 3 provide concise overviews of the simple continued fraction algorithm (Section 2) and the Chakravala method (Section 3) for solving the Pell equation. These sections serve to illustrate the natural extension of the algorithm, as we derived corresponding analogs to various equations presented therein (refer to Section 7). In Section 4, we introduce two novel algorithms that generalize the Chakravala method, accompanied by illustrative examples. Section 5 introduces several variants of these algorithms, incorporating the LLL-algorithm with rank 2 lattices.

The discovery of these algorithms dates back to September 2015. In 2018, we further developed the details and compiled this work, excluding Subsection \ref{5-M}. The initial version of this work was published in August 2023, as referenced in \cite{L}.

\vskip .8cm

\section{Simple continued fractions and Pell equation}

\vskip .3cm

To create a self-contained work, in this section, we introduce the concept of a simple continued fraction and its application in solving the Pell equation. The subsequent definitions and results are standard, as found in \cite{Hua}, \cite{NZ}, \cite{JW}, and \cite{Joy}. Let $\lfloor x\rfloor$ denote the greatest integer less than or equal to the real number $x$.

The \textit{simple continued fraction expansion} of a real number $\phi$ (which we denote as $\phi_0$) is an expression of the form
\begin{equation}\label{1}
\phi = \phi_0 = q_0 + \cfrac{1}{q_1 + \cfrac{1}{q_2 + \cfrac{1}{q_3 + \ddots}}}
\end{equation}

\vskip .2cm

\noindent expressed as $\phi_0 = [q_0; q_1, q_2, q_3, \dots]$, where $q_0 = \lfloor \phi_0 \rfloor$, and the numbers $q_n$ for $n > 0$ are positive integers defined recursively by
\begin{equation}\label{PPHI}
\phi_{n+1} = \frac{1}{\phi_n - q_n}, \quad q_{n+1} = \lfloor \phi_{n+1} \rfloor \quad \text{for all } n \geq 0,
\end{equation}

\vskip .2cm

\noindent where $\phi_n$ is the $n${\it -th complete quotient}. Note that $\phi_0 = [q_0; q_1, \dots, q_n, \phi_{n+1}]$.

We define two sequences of integers recursively as $A_0 = q_0$, $B_0 = 1$, $A_1 = q_1 q_0 + 1$, and $B_1 = q_1$. Additionally, we have
\begin{align}\label{3}
A_{n+1} &= q_{n+1} A_n + A_{n-1}\\
B_{n+1} &= q_{n+1} B_n + B_{n-1} \quad \text{for } n \geq 1.\nonumber
\end{align}

It is straightforward to establish by induction (refer to Theorem 7.5 and Theorem 7.4 of \cite{NZ}) that:
\begin{equation}\label{4}
A_n B_{n-1} - B_n A_{n-1} = (-1)^{n-1},
\end{equation}
and
\begin{equation*}
\frac{A_n}{B_n} = [q_0; q_1, q_2, q_3, \dots q_n] = q_0 + \cfrac{1}{q_1 + \cfrac{1}{\ddots + \cfrac{1}{q_{n-1} + \frac{1}{q_n}}}}
\end{equation*}

\vskip .2cm

\noindent The quotient ${A_n}/{B_n}$  is denoted as the
 $n${\it -th convergent} of the continued fraction (\ref{1}).

We now focus on the simple continued fraction expansion of $\sqrt{d}$ for a positive integer $d$ that is not a perfect square. The simple continued fraction expansion for $\sqrt{d}$ provides all the necessary tools to solve Pell's equation $x^2 - dy^2 = 1$, as demonstrated by Euler and Lagrange. The following result is derived from Theorem 8.1 in \cite{Hua} or page 346 in \cite{NZ}.

\begin{proposition}\label{prop 2}
Let $d$ be a positive integer and not a perfect square. In the
continued fraction of $\phi_0=\sqrt{d}$, we have
\begin{equation}\label{7000}
    \phi_n=\frac{P_n+\sqrt{d}}{Q_n}, \qquad
    \hbox{ for }n\geq 0
\end{equation}

\vskip .2cm

\noi where $P_n$ and $Q_n$ are the  integers defined
recursively by $P_0=0$, $Q_0=1$ and
\begin{equation}\label{8}
    P_{n+1}=q_n Q_n-P_n, \qquad\hbox{ and } \qquad
    Q_{n+1}=\frac{d-P_{n+1}^2}{Q_n}\quad \hbox{ for } n\geq 0,
\end{equation}
with $q_0=\lfloor \sqrt{d}\rfloor$ and
    \begin{equation}\label{vvver}
q_n=\Bigg\lfloor\frac{P_n+\lfloor\sqrt{d}\rfloor}{Q_n}\Bigg\rfloor.
\end{equation}
 We also have that
\begin{equation}\label{99}
    0<Q_n<2\sqrt{d},\qquad 0<P_n<\sqrt{d}.
\end{equation}
\end{proposition}

\vskip .1cm

In the subsequent proposition, we present some valuable identities as they offer a clear connection between the simple continued fraction, the Chakravala method, and the generalizations provided in Section 4 (refer to equations (\ref{3.9}) and (\ref{4.9}), as well as Section \ref{7ma}).

\

\begin{proposition} (p.46 in \cite{JW})\label{2.2.2}
Let $d$ be a positive integer and not a perfect square. In the
continued fraction of $\phi_0=\sqrt{d}$,  the following relationships hold:
\begin{equation}\label{17}
    A_n^2-dB_n^2=(-1)^{n+1}Q_{n+1}
\end{equation}
and
\begin{align}
    A_n= \, \frac{P_{n+1}A_{n-1}+dB_{n-1}}{Q_n},   \qquad\quad
    B_n= \, \frac{A_{n-1}+P_{n+1}B_{n-1}}{Q_n},\label{18}
\end{align}
along with
\begin{align}\label{19}
    dB_n &= P_{n+1}A_n+Q_{n+1}A_{n-1}\\
    A_n &=P_{n+1}B_n+Q_{n+1}B_{n-1}.\nonumber
\end{align}
\end{proposition}

\

The continued fraction algorithm stops when $A_n^2 - dB_n^2 = 1$, i.e., when $(-1)^{n+1} Q_{n+1} = 1$ (see (\ref{17})).
The \textit{fundamental solution} of $x^2 - y^2 d = 1$ is the solution $(X, Y)$ in the smallest positive integers, denoted as $\epsilon = X + Y\sqrt{d}$.
The number $R_{\, d}^{^{_{10}}} = \log_{10}(X + Y\sqrt{d})$ is referred to as the \textit{regulator (with base 10)}. We have used logarithm with base 10, even if it is not usual in number theory. We only use this notation in Section \ref{lL}.

\begin{theorem} (Theorem 7.26  of \cite{NZ})
Let $d$ be a positive integer, not a perfect square. If $\epsilon=X+Y\sqrt{d}$ is
the fundamental solution  of $x^2-d y^2=1$, then all positive
solutions are given by $x_n,y_n$, for $n=1,2,\dots$, where $x_n$ and
$y_n$ are the integers defined by
\begin{equation*}
x_n+y_n \sqrt{d}=(X+Y \sqrt{d})^n.
\end{equation*}
\end{theorem}

\

Observe that  the equations in (\ref{18}) can be rewritten
as
\begin{equation}\label{pro cont}
A_n+\sqrt{d} \,B_n=(A_0+\sqrt{d} B_0) \  \prod_{j=1}^n\ \bigg(
\frac{P_{j+1}+\sqrt{d}\,}{Q_j} \bigg).
\end{equation}
where $P_i$ and $Q_i$ are defined in Proposition \ref{prop 2}. This is similar to what is referred to as the \textit{"power product"} in \cite{Len}.
Therefore, we have a  product representation of the convergent
and the fundamental solution of the Pell equation.

\

\section{Chakravala or cyclic method}\lbb{3ra}

\

The Indian mathematician Bhaskara II described the first method to solve the Pell equation, known as the \textit{Chakravala (or cyclic) method}, specifically addressing the case $x^2 - 61y^2 = 1$ (among other examples). Here, we present a variant of the algorithm, considering the version provided in \cite{Bar}, \cite{Ed}, and \cite{wiki}. (Note: There are several variants depending on the signs; refer to Remark \ref{signo} for additional details.)

As mentioned in the introduction, in the Chakravala method, the primary concept involves taking a triple $(a, b, k)$ satisfying $a^2 - db^{\, 2} = k$ and composing it with the trivial triple $(m, 1, m^2 - d)$ to obtain a new triple $(am + db, a + bm, k(m^2 - d))$. This can be scaled down by $k$ to yield
\begin{equation}\label{3.4b}
\left(\frac{am + db}{k}\right)^2 - d\left(\frac{a + bm}{k}\right)^2 = \frac{m^2 - d}{k}.
\end{equation}
We then choose $m$ as a positive integer such that $a + b\,m$ is divisible by $k$ and minimizes the absolute value of $m^2 - d$, and hence that of $(m^2 - d)/k$. Assuming $(a, b) = 1$ (and therefore $(k, b) = 1$), we observe that $m^2 - d$ and $am + db$ are also multiples of $k$ using the following equations:
\begin{equation}\label{3.33}
(m^2 - d)b^2 = k - (a^2 - m^2b^2) = k - (a + mb)(a - mb)
\end{equation}
and
\begin{equation}\label{3.34}
am + db = (a + mb)m - (m^2 - d)b.
\end{equation}
This leads to a new triple of \textit{integers}
\begin{equation*}
\tilde{a} = \frac{am + db}{|k|}, \quad \tilde{b} = \frac{a + bm}{|k|}, \quad
\tilde{k} = \frac{m^2 - d}{k},
\end{equation*}
satisfying $\tilde{a}^2 - d\tilde{b}^2 = \tilde{k}$. The process continues until a stage is reached at which the equation has the desired form $a^2 - db^2 = 1$, i.e., $\tilde{k} = 1$.

More precisely, given a  non-square positive integer $d$, the
algorithm produces sequences of integers $A_i, B_i, Q_i$, and $ P_i$
according to the following recipe: we start with $A_0=1,B_0=0,Q_1=1,P_1=0$.
Given integers $A_{n-1},B_{n-1},Q_n$ and $P_n$ where
$(A_{n-1},B_{n-1})=1$ such that
\begin{equation*}
    A_{n-1}^2-d B_{n-1}^2=Q_n,
\end{equation*}
we choose $P_{n+1}$ to be a positive integer for which
$A_{n-1}+B_{n-1}P_{n+1}$ is divisible by $Q_n$ and minimizes the
absolute value of $P_{n+1}^2-d$. Then we take (cf. (\ref{8}))
\begin{equation}\label{3.8}
Q_{n+1}=\frac{P_{n+1}^2-d}{Q_n}
\end{equation}
and (cf. (\ref{18}))
\begin{equation}\label{3.9}
    A_n=\frac{A_{n-1}P_{n+1}+d B_{n-1}}{|Q_n|},\qquad B_n=
    \frac{A_{n-1}+B_{n-1}P_{n+1}}{|Q_n|}.
\end{equation}
Using (\ref{3.33}), (\ref{3.34}), and (\ref{3.8}), we obtain that
$A_n$ and $Q_{n+1}$ are integers. By using (\ref{3.4b}) we get
(cf.(\ref{17}))
\begin{equation}\label{3.10}
    A_n^2-d B_n^2=Q_{n+1},
\end{equation}

\noindent and $(A_n,B_n)=1$ is obtained by observing that $|A_nB_{n-1}-B_n A_{n-1}|=1$.
The method terminates when $Q_{n+1}=1$ for some $n$, and it is
possible to show that in this case, $A_{n}+B_{n} \sqrt d$ is the
fundamental solution of the Pell equation. We shall not prove it
here, see p.35 in \cite{Ed}.

\vskip .2cm

\begin{remark}\label{signo}
On page 33 of \cite{JW}, and in \cite{Bau}, another version of the Chakravala method is employed, where all the elements $Q_n$ are always positive integers. Specifically, given integers $A_{n-1}$, $B_{n-1}$, $Q_n$, and $P_n$ such that $(A_{n-1},B_{n-1})=1$ and satisfying
\begin{equation*}
    |A_{n-1}^2-d B_{n-1}^2|=Q_n,
\end{equation*}
they select $P_{n+1}$ to be a positive integer such that $A_{n-1}+B_{n-1}P_{n+1}$ is divisible by $Q_n$ and minimizes the absolute value of $P_{n+1}^2-d$. Then, they define
\begin{equation*}
Q_{n+1}=\frac{|P_{n+1}^2-d|}{Q_n}
\end{equation*}
and
\begin{equation*}
    A_n=\frac{A_{n-1}P_{n+1}+d B_{n-1}}{Q_n},\qquad B_n=
    \frac{A_{n-1}+B_{n-1}P_{n+1}}{Q_n}.
\end{equation*}

\vskip .2cm

\noindent As mentioned earlier, this procedure ensures that $A_n$ and $Q_{n+1}$ are integers, and
\begin{equation*}
   | A_n^2-d B_n^2|=Q_{n+1}.
\end{equation*}

\noindent
The method terminates when, for some $n$, $Q_{n+1}=1,2,4$ (i.e. $A_n^2-d B_n^2=\pm 1,\pm 2,\pm 4$). At this point, Brahmagupta's composition method is used to construct the fundamental solution. However, we will not delve into the details of this method here (see \cite{wiki} for further information).
\end{remark}

\begin{remark} \label{inicio}
{\bf The first step in the Chakravala algorithm:} We choose $P_2$
as a positive integer to minimize $|P_2^2-d|$. Consequently, $P_{ 2}$ can take the value of
 $\lfloor\sqrt{d}\rfloor$ or $\lfloor\sqrt{d}\rfloor+1$. Therefore

\vskip .2cm

$\ast$ if
$
(\lfloor\sqrt{d}\rfloor+1)^2-d<d-\lfloor\sqrt{d}\rfloor^2
$
, then $P_2=\lfloor\sqrt{d}\rfloor+1, A_1=\lfloor\sqrt{d}\rfloor+1,
B_1=1,Q_2=A_1^2-d\, B_1^2$,

$\ast$ else: $P_2=\lfloor\sqrt{d}\rfloor,
A_1=\lfloor\sqrt{d}\rfloor, B_1=1,Q_2=A_1^2-d\, B_1^2$.
\end{remark}

\vskip .3cm

\noindent From this point onward in the examples, \textbf{we will start from this initial condition}.

\vskip .3cm

By induction and using the ideas of the proof of Proposition 1 in \cite{Bau}, it is possible to demonstrate the following result, cf. (\ref{99}).

\vskip .3cm

\begin{proposition}
The sequence of integers $Q_n$ satisfies
\begin{equation}\label{3.11}
|Q_n|<\sqrt{d} \quad \hbox{ for all }\  n\geq 1.
\end{equation}
\end{proposition}

By utilizing (\ref{3.11}) and  (\ref{3.10}), in conjunction with the following proposition, we establish that the quotients
$A_n/B_n$ generated by the Chakravala method are indeed convergent in the simple continued fraction expansion of $\sqrt{d}$.

\vskip .3cm

\begin{proposition}\label{pp}
{\rm (Theorem 7.24 in \cite{NZ})} If $A$ and $ B$ are positive integers
with $(A,B)=1$, $d$ is a  non-square positive integer and $Q\in
\mathbb{Z}$ such that $|Q|< \sqrt{d}$ and
\begin{equation*}
A^2-d B^2=Q,
\end{equation*}
then there is an $i>0$ for which $A=A_i$ and $B=B_i$, where $A_i/B_i$
is the $i$-th convergent in the simple continued fraction
expansion of $\sqrt{d}$.
\end{proposition}

\vskip .2cm

The exercises on page 35 of \cite{Ed} show that the cyclic method efficiently finds the Pell equation's fundamental solution, often skipping non-solution steps inherent in the continued fraction method. Computational evidence suggests the Chakravala method needs roughly 69\% of the steps required by the continued fraction approach.

\vskip .2cm

Now, we present an implementation of the algorithm following p.34 in \cite{JW}. The goal is to transform the conditions on $P_{n+1}$, namely $A_{n-1}+B_{n-1}P_{n+1}$ being divisible by $Q_n$ and $|P_{n+1}^2-d|$ being minimal, into simpler ones, avoiding the use of the large numbers $A_{n-1}$ and $B_{n-1}$.

Using (\ref{3.9}) and (\ref{3.8}), the following expression is derived:
\begin{align}\label{909}
P_n B_{n-1} - A_{n-1} &= \frac{P_n(A_{n-2} + B_{n-2}P_n) - (A_{n-2}P_n + d B_{n-2})}{|Q_{n-1}|} = B_{n-2} \frac{(P_n^2 - d)}{|Q_{n-1}|}   \\
&= B_{n-2} Q_n \text{ sign}(Q_{n-1}).\nonumber
\end{align}

\vskip .2cm

\noindent
Hence, we have that $P_n B_{n-1} - A_{n-1} \equiv 0 \ \ ({\rm mod }\ |Q_n|)$. Also, by construction, $Q_n | (P_{n+1} B_{n-1} + A_{n-1})$. Therefore, $(P_{n+1} + P_n) B_{n-1} \equiv 0 \ \ ({\rm mod }\ |Q_n|)$. Since $(Q_n, B_{n-1}) | (A_{n-1}, B_{n-1})$, it is obtained that $(Q_n, B_{n-1}) = 1$, and
\begin{equation}\label{3.12}
P_{n+1} \equiv -P_n \ \ ({\rm mod}\ |Q_n|).
\end{equation}

\noindent
Hence, using the fact that $P_j$'s are positive integers, the following relation holds:
\begin{equation*}\label{KEY}
P_{n+1} = -P_n + q_n |Q_n|
\end{equation*}

\noindent
for some positive integer $q_n$  (cf. (\ref{8})). It should be noted that $P_{n+1}$ must satisfy $|P_{n+1}^2-d| \leq |P^2-d|$ for any positive $P$ congruent to $-P_{n}$ modulo $|Q_n|$. By Theorem 2.2 on page 34 in \cite{JW}, at each step, the following expression is chosen (cf. (\ref{vvver})):
\begin{equation*}\label{bb}
q = \bigg\lfloor \frac{P_n + \sqrt{d}}{|Q_n|} \bigg\rfloor,
\end{equation*}
and the possible values of $q_n$ are $q$ or $q+1$, obtaining $P_{n+1}$. Now, $A_n, B_n$, and $Q_{n+1}$ are defined as in (\ref{3.8}) and (\ref{3.9}). If $Q_{n+1} = 1$, then the algorithm is stopped, obtaining the fundamental solution.

\vskip .3cm

Observe that in the continued fraction algorithm, $q_n$ in (\ref{vvver}) corresponds to selecting $P_{n+1}$ as the integer congruent to $-P_n$ modulo $Q_n$ (where all $Q_n$ are positive) such that $\sqrt{d} - P_{n+1}$ is \textit{positive} and minimum (refer to (\ref{8}) and Section \ref{7ma}), or equivalently, such that ${d} - P_{n+1}^2$ is \textit{positive} and minimum. Similar to the continued fraction algorithm (see Proposition \ref{prop 2}), we define the numbers $P_{n+1}$ and $Q_{n+1}$ without involving $A_{n-1}$ and $B_{n-1}$. Finally, by using this implementation of the Chakravala method, we can express the fundamental solution as a power product, akin to the situation observed in (\ref{pro cont}).

\

\vskip .5cm

\section{Two algorithms that generalize the Chakravala method, that incorporate an additional variable}\lbb{4ta}

\

In this section, we present one of the main results of this work: two algorithms that generalize the Chakravala method. In Subsection \ref{firstL} we define the \textit{First Algorithm with $L$}, and in  Subsection \ref{secondL} we define the \textit{Second Algorithm with $L$}.

\vskip .2cm

The basic idea for both algorithms is very simple: in the Chakravala method,  we
composed a  triple $(a,b,k)$  with the trivial triple
$(m,1,m^2-d)$ (that is, we put $l=1$ in (\ref{3.3})) to get a new
triple $(am+db,a+bm,k(m^2-d))$ which can be scaled down by $k$ as  explained in the Introduction and Section 3.

However, we can repeat the entire argument with the "$l$" included. More precisely,  as it was pointed out in (\ref{3.2}) and (\ref{3.3}), given a
triple $(a,b,k)$ (that satisfies $a^2-db^2=k$), we
can compose it with the  triple $(m,l,m^2-d\,l^2)$ (see
(\ref{3.3})) to get a new triple $(am+db\,l,al+bm,k(m^2-dl^2))$
which can be scaled down by $k$ to get
\begin{equation}\label{4.4}
\bigg(\frac{am+db\,l}{k}\bigg)^2-d\bigg(\frac{al+b\,m}{k}\bigg)^2
=\frac{m^2-dl^2}{k}.
\end{equation}
We now  have  to choose two variables, $m$ and $l$. First of all,
we choose $m$ and $l$ to be  positive integers for which $al+b\,m$
is divisible by $k$. Then, assuming that $(a,b)=1$ (and therefore
$(k,b)=1$),
 we  see that $m^2-dl^2$
is also a multiple of $k$. This is established by using the equation:
\begin{equation}\label{4.33}
(m^2-dl^2)b^2=kl^2-(a^2l^2-b^2m^2)=kl^2-(al+bm)(al-bm).
\end{equation}

\vskip .2cm

\noindent
Furthermore, using (\ref{4.4}), we  have that  $am+db\,l$ is also a multiple
of $k$. This results in a new triple of {\it integers}
\begin{equation*}
    \tilde{a}=\frac{am+db\,l}{|k|}, \quad \tilde{b}=\frac{al+b\,m}{|k|}, \quad
    \tilde{k}=\frac{m^2-dl^2}{k},
\end{equation*}
where $\tilde{a}^2-d\tilde{b}^2=\tilde{k}$.

\

Now, we are free to impose an additional condition of minimization, as in the Chakravala method.

\vskip .4cm

\subsection{First Algorithm with L}\lbb{firstL}

In this first case, we  impose that $m$ and $l$ are chosen to minimize the absolute value of $m^2 - dl^2$ and, consequently,  that
of $(m^2-dl^2)/k$, with $1\leq l\leq L$ for some fixed $L$.
The process continues until a stage is reached at
which the equation has the desired form $a^2-db^2=1$.

More precisely, given a  non-square positive integer $d$ and a
positive integer $L$, the algorithm produces sequences of
integers $a_i, b_i, k_i, m_i$, and $ l_i$ using the following recipe. We start with  the first step in the Chakravala algorithm, from Remark \ref{inicio}:
\begin{align}\label{inicio2}
& \mathrm{if \ \ }(\lfloor\sqrt{d}\rfloor+1)^2-d<d-\lfloor\sqrt{d}\rfloor^2
,\mathrm{\ \ then\ \ }  a_1=\lfloor\sqrt{d}\rfloor+1,\,
b_1=1,\,k_2=a_1^2-d\, b_1^2, \\
& \mathrm{else:\ \  } a_1=\lfloor\sqrt{d}\rfloor, \, b_1=1,\,k_2=a_1^2-d\, b_1^2.\nonumber
\end{align}

\vskip .2cm

\noindent
Now, given integers $a_{i-1},b_{i-1}$ and $k_i$ where
$(a_{i-1},b_{i-1})=1$ such that
\begin{equation}\label{new}
    a_{i-1}^2-d \, b_{i-1}^2=k_i,
\end{equation}
we choose $m_{i+1}$ and $l_{i+1}$ as positive integers for
which $a_{i-1}l_{i+1}+b_{i-1}m_{i+1}$ is divisible by $k_i$ and
minimizes the absolute value of $m_{i+1}^2-d\,l_{i+1}^2$ for
$1\leq l_{i+1}\leq L$. Then we take (cf. (\ref{3.8}))
\begin{equation}\label{4.8}
k_{i+1}=\frac{m_{i+1}^2-d\,l_{i+1}^2}{k_i}
\end{equation}
and (cf. equations (\ref{18}) and ({\ref{3.9}))

\begin{equation}\label{4.9}
    a_i=\frac{a_{i-1}m_{i+1}+d \, b_{i-1}l_{i+1}}{|k_i|},\qquad b_i=
    \frac{a_{i-1}l_{i+1}+b_{i-1}m_{i+1}}{|k_i|}.
\end{equation}

\vskip .2cm

\noindent
As before, using (\ref{4.33}), (\ref{4.4}) and $(a_{i-1},b_{i-1})=1$, we obtain that $k_{i+1}$ and $a_i$
are integers. By  (\ref{4.4}) we get (cf.
(\ref{17}) and (\ref{3.10}))
\begin{equation*}
    a_i^2-d \, b_i^2=k_{i+1},
\end{equation*}

\vskip .2cm

\noindent
and  to complete the recursive definition, we need to
prove that $(a_i,b_i)=1$, see  Proposition {\ref{p.1}}. The method
terminates when $k_n=1$ for some $n$. \textit{In the examples, we always obtained the fundamental solution}, except for extremely   large values of $L$ relative to $d$ (cf. the analysis in Table \ref{table:40}).
We refer to this algorithm as the  {\bf First Algorithm with L}.

\vskip .5cm

\begin{proposition}\label{p.1}
The integers $a_i$ and $b_i$ defined by (\ref{4.9}) are coprime.
\end{proposition}

\vskip .1cm

\begin{proof}
Using (\ref{4.9}) and (\ref{new}), we have:
\begin{align*}
  a_i \, b_{i-1} - b_i\,  a_{i-1}
  & = \frac{(a_{i-1}m_{i+1}+d \, b_{i-1}l_{i+1})\, b_{i-1} - (a_{i-1}l_{i+1}+b_{i-1}m_{i+1})\, a_{i-1}}{|k_i|}\\
  & = \frac{(d b_{i-1}^2-a_{i-1}^2)}{|k_i|}\ l_{i+1}
  = - \frac{k_i}{|k_i|}\ l_{i+1}.
\end{align*}
Hence, we obtain that $|a_ib_{i-1}-b_i
a_{i-1}|=l_{i+1}$. If we define $h=(a_i,b_i)$, then $h|\,l_{i+1}$. Using this
 together with (\ref{4.9}), we obtain that $h| (b_{i-1} m_{i+1})$
 and $h| (a_{i-1} m_{i+1})$. But $(a_{i-1},b_{i-1})=1$, hence $h|
 m_{i+1}$.
 Suppose
  $h>1$, and  we will demonstrate that this assumption leads to a contradiction.
 Using all the previous results, we can write $l_{i+1}=h \,\tilde
 l_{i+1}, m_{i+1}=h \,\tilde m_{i+1}$, and $b_{i}=h \,\tilde
 b_{i}$. Since $b_i |k_i|=a_{i-1}l_{i+1}+b_{i-1}m_{i+1}$, then
$\tilde b_i |k_i|=a_{i-1}\, \tilde l_{i+1}+b_{i-1}\,\tilde
m_{i+1}$, and therefore, the positive integers $\tilde m_{i+1}$ and
$\tilde l_{i+1}$ satisfy that $a_{i-1}\,\tilde
l_{i+1}+b_{i-1}\,\tilde m_{i+1}$ is divisible by $k_i$ and:
\begin{equation*}
|m_{i+1}^2-d\,l_{i+1}^2|=h^2\,|\tilde m_{i+1}^2-d\,\tilde
l_{i+1}^{\ 2}|> |\tilde m_{i+1}^2-d\,\tilde l_{i+1}^{\ 2}|
\end{equation*}
which is a contradiction due to the minimization property
satisfied by the pair $m_{i+1}$ and $l_{i+1}$, finishing the
proof.
\end{proof}

Note that this algorithm, with $L=1$ (or $l_i=1$ for all $i>0$)
corresponds to the Chakravala algorithm. Based on examples, we observe  that
(\ref{3.11}) holds,  that is
\begin{equation}\label{720}
    |k_i|< \sqrt{d} \qquad \hbox{ for all } i>0,
\end{equation}
and, using Proposition \ref{pp}, we conclude that in this algorithm,  $a_i / b_i$ are convergent in the simple continued fraction expansion of $\sqrt{d}$.

\vskip .2cm

\subsection{Second Algorithm with L}\lbb{secondL}

Now, we present a second version of this algorithm by changing the additional condition of minimization. In
our second case, we choose to impose that $m$ and $l$ are positive
integers chosen such that they minimize the absolute value of
$m-\sqrt{d}\, l$, with $1\leq l\leq L$ for some fixed $L$.
The process is continued until a stage is reached at
which the equation has the desired form $a^2-db^2=1$, { under
certain conditions on $L$ that depend on $d$} (cf. Table \ref{table:40}).

\ 

More precisely, given a  non-square positive integer $d$ and a
positive integer $L$, the algorithm produces  sequences of
integers $a_i, b_i, k_i, m_i$, and $l_i$ by the following recipe: as
before, we prefer to start with the first step in the Chakravala
algorithm  as in Remark \ref{inicio} and (\ref{inicio2}):
\begin{align}\label{star}
& \mathrm{if \ \ }(\lfloor\sqrt{d}\rfloor+1)^2-d<d-\lfloor\sqrt{d}\rfloor^2
,\mathrm{\ \ then\ \ }  a_1=\lfloor\sqrt{d}\rfloor+1,\,
b_1=1,\,k_2=a_1^2-d\, b_1^2, \\
& \mathrm{else:\ \  } a_1=\lfloor\sqrt{d}\rfloor,\, b_1=1,\,k_2=a_1^2-d\, b_1^2.\nonumber
\end{align}
Now, given integers $a_{i-1},b_{i-1}$ and $k_i$ where
$(a_{i-1},b_{i-1})=1$ such that
\begin{equation*}
    a_{i-1}^2-d b_{i-1}^2=k_i,
\end{equation*}
we choose $m_{i+1}$ and $l_{i+1}$ to be a positive integers for
which $a_{i-1}l_{i+1}+b_{i-1}m_{i+1}$ is divisible by $k_i$ and
minimizes the absolute value of $m_{i+1}-\sqrt{d}\ l_{i+1}$ for
$1\leq l_{i+1}\leq L$. Then we take $k_{i+1},a_i$ and $b_i$ as in
 (\ref{4.8}) and (\ref{4.9}).

As before, using (\ref{4.33}), (\ref{4.4}) and $(a_{i-1},b_{i-1})=1$, we obtain that $k_{i+1}$ and $a_i$
are integers. By using (\ref{4.4}) we get
\begin{equation*}
    a_i^2-d b_i^2=k_{i+1},
\end{equation*}
and in order to complete the recursive definition, we need to
prove that $(a_i,b_i)=1$. However, this follows immediately from the same
arguments given in the proof of  Proposition {\ref{p.1}}. The
method terminates when $k_n=1$ for some $n$.  \textit{In all the examples, we obtain the fundamental solution}, except for extremely large values of $L$, see the analysis in Table \ref{table:40}.
\begin{claim} {\rm (Based on examples)}
In this algorithm, $a_i/b_i$ are also convergent in
the simple continued fraction expansion of $\sqrt{d}$, and
  $  |k_i|< 2\sqrt{d}  \hbox{ for all } i>0$.
\end{claim}

Observe that this algorithm is not equivalent to
Chakravala when restricted to the case $L=1$.

\subsection{Examples}\lbb{Ex}

Now, we present some illustrative examples. In Table \ref{table:1}, the initial example
corresponds to $d=61$, where we display the sequences (generated by
the four algorithms) of three terms $(a,b,k)$ such that
$a^2-db^2=k$, until reaching the fundamental solution of the Pell equation:

\begin{table}[H]
\caption{Computing the fundamental solution when $d=61$}
\label{table:1} \
\begin{center}
\fontsize{8}{11}\selectfont
\begin{tabular}{ | c | c | c | c |}
\hline
Cont. Frac.& Chakravala & $1^{st}$ algorithm with $L=9$ & $2^{nd}$ algorithm with $L=9$ \\
\hline \hline
(7, 1, -12) &     &   &  \\
(8, 1, 3) &   (8, 1, 3)   & (8, 1, 3)   & (8, 1, 3)  \\
(39, 5, -4) &   (39, 5, -4)  &  (39, 5, -4) &  \\
(125, 16, 9) &     &   &  \\
(164, 21, -5) &  (164, 21, -5)   & (164, 21, -5)  & (164, 21, -5) \\
(453, 58, 5) &  (453, 58, 5)   & &  \\
(1070, 137, -9) &     &   &  \\
(1523, 195, 4) &   (1523, 195, 4)  &  (1523, 195, 4) & (1523, 195, 4) \\
        (5639, 722, -3) &  (5639, 722, -3)   &   &  \\
        (24079, 3083, 12) &     &   &  \\
(29718, 3805, -1) &  (29718, 3805, -1)& (29718, 3805, -1)  &  (29718, 3805, -1)\\
        (440131, 56353, 12)&     &   &  \\
        (469849, 60158, -3)&  (469849, 60158, -3)   & (469849, 60158, -3)  &  \\
(2319527, 296985, 4)&  (2319527, 296985, 4)& (2319527, 296985, 4)& (2319527, 296985, 4) \\
        (7428430, 951113, -9)&     &   &  \\
(9747957, 1248098, 5)&   (9747957, 1248098, 5)  &  (9747957, 1248098, 5) &  \\
        (26924344, 3447309, -5)&  (26924344, 3447309, -5)   &   &  \\
        (63596645, 8142716, 9)&     &   & (63596645, 8142716, 9) \\
(90520989, 11590025, -4)&  (90520989, 11590025, -4)   & (90520989, 11590025, -4)  &  \\
(335159612, 42912791, 3)&  (335159612, 42912791, 3)   &   & (335159612, 42912791, 3) \\
        (1431159437, 183241189, -12)&     &   &  \\
 (1766319049, 226153980, 1)\!  &  (1766319049, 226153980, 1)\!
  &  (1766319049, 226153980, 1)\!  & (1766319049, 226153980, 1)\!  \\
\hline
\end{tabular}
\end{center}
\end{table}

Observe,  that we need 22 steps with the continued fraction, 14 steps
with Chakravala, and only 10 and 8 steps in both algorithms with
$L=9$, respectively. We always obtained convergent of the
continued fraction of $\sqrt{61}$. Note that in Chakravala and
the First Algorithm with $L=9$, we have $|k_i|<\sqrt{61}$, but in the other two algorithms, we have
$|k_i|<2\sqrt{61}$.

In Table \ref{table:2}, we present the number of steps  required  for
each value of $d$ in each algorithm.
\begin{table}[H]
\caption{Number of steps  required  for each value of $d$}
\label{table:2} \
\begin{center}
\fontsize{11}{16}\selectfont
\begin{tabular}{| c || c | c | c | c |}
\hline
  d &   Cont. Frac.& Chakravala &  $1^{st}$ algorithm with $L=9$ & $2^{nd}$ algorithm with $L=9$ \\
\hline \hline
 46 &  12 &   8  & 4 &  4  \\
 61 &  22 &  14  & 10 & 8  \\
 97 &  22 & 12    & 8 & 6   \\
 109 &  30 &  22  & 15 & 11   \\
 313 &  34 &  26  & 14  &  14\\
 541 &  78 &  56  & 32 & 27  \\
\hline
\end{tabular}
\end{center}
\end{table}

Based on examples, when using $L=9$ in the Second Algorithm, the number of steps needed is
$0.34$ times the number of steps using continued fractions. Also,
with $L=100$, it is $0.2$ times the number of steps.

We also observe an asymptotic trend in the number of steps in the following example. Taking $d=132901$, in this case,
we need 422 steps with the continued fractions algorithm. The table below illustrates the steps for different values of $L$ with the Second Algorithm:

\begin{table}[H]
\caption{Number of steps with $L$ when $d=132901$ }
\label{table:40} \
\begin{center}
\fontsize{10}{16}\selectfont
\begin{tabular}{ | c || c | c | c | c | c | c | c | c | c |}
\hline
 L & 9 & 100 & 1000 & 1500 & 1625 & 1687
 & 1698  & 1699 & $\geq$ 1700  \\
\hline \hline  steps & 141 & 83 & 63 & 60 & 60 & 59
& 59 & 59  & diverges   \\
   \hline
\end{tabular}
\end{center}
\end{table}
\noi In this case, "diverges" and "$\geq$ 1700" mean that for
several values of $L$ greater or equal to 1700, the algorithm does
not stop after more than 300 steps.

\

\vskip .5cm

\section{Implementation and improvements of both algorithms}\lbb{5ta}

\

The other main results of this research pertain to the implementations and enhancements of the Second Algorithm with L. Notably, the \textit{Second Algorithm with LLL}, as detailed in Subsection \ref{lL}, stands out as the most significant.
As in the implementation of the Chakravala algorithm presented in Section 3,  the aim of this section is to transform the conditions on $m_{i+1}$ and $l_{i+1}$ into simpler ones, avoiding the use of the very large numbers $a_{i-1}$ and $b_{i-1}$. Observe that in both algorithms, we have to find at each step $m_{i+1}$ and $l_{i+1}$ satisfying that $k_i$ divides $a_{i-1} l_{i+1}+b_{i-1} m_{i+1}$. In Subsection \ref{5-1} and Subsection \ref{5-M}, we present an equivalent condition without involving $a_{i-1}$ and $b_{i-1}$. In Subsection \ref{5-2}, we study the minimization condition in the First Algorithm with L. In Subsection \ref{5-3}, we describe the minimization condition in the Second Algorithm with L, and  we give the implementation and some improvements of the Second Algorithm with L, presenting two variants of it. In Subsection \ref{lL}, we define the Second Algorithm with LLL.

\subsection{An equivalent condition for $m_{i+1}$ and $l_{i+1}$ such that $k_i$ divides $a_{i-1} l_{i+1}+b_{i-1} m_{i+1}$}\lbb{5-1}

The following  results hold for both
algorithms with L. Since $(a_{i-1},b_{i-1})=1$, we have that
$(b_{i-1},k_i)=1$. Therefore, the condition that $k_i$ divides $a_{i-1} l_{i+1}+b_{i-1} m_{i+1}$ is  \textbf{equivalent} to
\begin{equation}\label{4.51}
    m_{i+1}\equiv -\frac{\ a_{i-1}}{\ b_{i-1}}\ l_{i+1} \ \ {\hbox{ (mod
    $|k_i|$)}}.
\end{equation}
Now, we define $M_i$ such that $0\leq M_i<|k_i|$ and
\begin{equation}\label{4.52}
    M_{i}\equiv -\frac{\ a_{i-1}}{\ b_{i-1}}  \ \ {\hbox{ (mod
    $|k_i|$)}}.
\end{equation}
Hence, we have that $k_i$ divides $a_{i-1} l_{i+1}+b_{i-1} m_{i+1}$ \textbf{if and only if}

\begin{equation}\label{4.53}
    m_{i+1}=M_i \,l_{i+1} + r_i \ |k_i|,
\end{equation}

\vskip .2cm

\noindent for some integer $r_i$. In  Subsection \ref{5-M}, we present how to get $M_i$ in terms of $m_j,l_j,k_j$, and  $M_{j-1}$ for $j\leq i$, that is, without involving the large numbers $a_{i-1}$ and $b_{i-1}$, and this can be used in the implementation of any version of the algorithm. Equation (\ref{4.53}) will be used to impose the minimization conditions of both algorithms.

\vskip .3cm

\subsection{Minimization Condition of the First Algorithm with L}\lbb{5-2}

Recall that in the First Algorithm with L we have to minimize  $|m_{i+1}^2-d\,l_{i+1}^2|$ for
$1\leq l_{i+1}\leq L$. Using (\ref{4.53}),
we have
\begin{align}\label{4.54}
    \big|{m_{i+1}^2-d\, l_{i+1}^2}\big| &
    ={|m_{i+1}-\sqrt{d}\, l_{i+1} |}
    {|m_{i+1}+\sqrt{d}\, l_{i+1} |} \\
    &  = |k_i|^2 \bigg|r_i-\bigg(\frac{\sqrt{d}-M_i}{|k_i|}\bigg)\,
    l_{i+1}\bigg|\bigg|r_i+\bigg(\frac{\sqrt{d}+M_i}{|k_i|}\bigg)\,
    l_{i+1}\bigg|.\nonumber
\end{align}
For each $l_{i+1}$ with $0<l_{i+1}\leq L$, we have two options for
$r_i$: $r_i^+$ as one of the integers near to
$\frac{(\sqrt{d}-M_i)}{|k_i|}\,l_{i+1}$ or $r_i^-$ as one of the
integers near to $\frac{(-\sqrt{d}-M_i)}{|k_i|}\,l_{i+1}$. If we
take $r_i^-$, then we have $m_{i+1}$ as an integer near to
$M_i l_{i+1}+|k_i|\frac{(-\sqrt{d}-M_i)}{|k_i|}l_{i+1}=-\sqrt{d}\,
l_{i+1}$ which is negative, contradicting our assumptions.

Therefore, we must take $r_i^+$, which is always positive because,
by (\ref{720}), we have $M_i<|k_i|<\sqrt{d}$. Thus, we clearly need to take
\begin{equation*}
r_i^+={\rm floor}\bigg(\frac{(\sqrt{d}-M_i)}{|k_i|}l_{i+1}\bigg)
\quad\hbox{ or } \quad r_i^+={\rm
ceil}\bigg(\frac{(\sqrt{d}-M_i)}{|k_i|}l_{i+1}\bigg).
\end{equation*}

\noindent Then, $l_{i+1}$ is the integer that produces the minimum
of (\ref{4.54}), for $1\leq l_{i+1}\leq L$.
Unfortunately, we do not know how to determine this minimum in a simple and fast way. \textit{With this, we conclude the analysis of the First Algorithm in this work.}

\vskip .5cm

\subsection{Implementation of the Second Algorithm with L}\lbb{5-3}

This subsection is one of the main parts of this work.  In each step of the Second Algorithm with
L, we need to choose $m_{i+1}$ and $l_{i+1}$ as positive integers for
which $a_{i-1}l_{i+1}+b_{i-1}m_{i+1}$ is divisible by $k_i$ and
  $|\,m_{i+1}-\sqrt{d}\ l_{i+1}|$ is minimal for
$1\leq l_{i+1}\leq L$. We have seen in
 (\ref{4.53}) that the condition $k_i\,|\,(a_{i-1}l_{i+1}+b_{i-1}m_{i+1})$ is equivalent to 
 $
 m_{i+1}=M_i \,l_{i+1} + r_i \ |k_i| $ 
 for some integer $r_i$, where $M_i$ can be computed without involving $a_{i-1}$ and $b_{i-1}$ as in Subsection \ref{5-M}.

Recall that in this algorithm $|k_i|<2\,\sqrt{d}$, but we obtain the following claim

\begin{claim} (Based on  examples) In this case, we also have $M_i<\sqrt{d}$ for all
$i>0$.
\end{claim}

Since $m_{i+1}-l_{i+1}\sqrt{d}=|k_i| r_i-l_{i+1} (\sqrt{d}-M_i)$,
we are looking for positive integers  $r$ and $l$ such that $ \big||k_i| r-l
(\sqrt{d}-M_i)\big| $ is minimum for $l$ from 1 to $L$ or
equivalently, that minimize $ \big|r- l
\big(\frac{\sqrt{d}-M_i}{|k_i|}\big)\big| $. Observe that for
each $l$, the integer $r$ is uniquely determined, that is $r= $\,
round$\Big(l\,\frac{(\sqrt{d}-M_i)}{|k_i|}\Big)$,  where "round"
is the closest integer.

\

\noindent \textbf{Summary of the implementation of the Second Algorithm with L without involving $\mathbf{a_i}$ and $\mathbf{b_i}$}.  We take  $a_1, b_1, k_2$ as in (\ref{star}). Then, we
define $l_3,m_3$ and $k_3$  using   Step 2 and Step 3,  with $M_2\equiv -\frac{a_1}{b_1}$
(mod $|k_2|$). Then, given $(l_{i},m_{i},k_i)$, we define
$(l_{i+1},m_{i+1},k_{i+1})$ as follows:

\vskip .2cm

\textbf{Step 1.} Compute
$M_i$ as in  Subsection \ref{5-M}.

\vskip .2cm

\textbf{Step 2.}
 Take $l_{i+1}$ as the integer that satisfies the following
minimization
\begin{equation}\label{ver}
\min_{l=1,\dots ,L}\ \Bigg|\,\ l\ \frac{(\sqrt{d}-M_i)}{|k_i|}\ -\
\hbox{round}\bigg(l\ \frac{(\sqrt{d}-M_i)}{|k_i|}\bigg)\ \Bigg|,
\end{equation}
\qquad\qquad \ \ \  and
\begin{equation}
  \label{rrrr}
r_{i}\   = \hbox{round}\bigg(l_{i+1}\ \frac{(\sqrt{d}-M_i)}{|k_i|}\bigg).
\end{equation}

\textbf{Step 3.} Define
\begin{align}
\label{mimi}
m_{i+1} & =M_i\,l_{i+1}\, +\, |k_i|\, r_{i},\\
k_{i+1} & =\frac{m_{i+1}^2-d\,l_{i+1}^2}{k_i}.\nonumber
\end{align}
It is clear, by using (\ref{rrrr}) and
(\ref{mimi}), that $m_{i+1}\simeq {l_{i+1}}
\sqrt{d}$, proving that $m_{i+1}$ is positive. The
method terminates when $k_n=1$ for some $n$.  In all the examples, we obtain the fundamental solution, except for extremely large values of $L$.

\vskip .2cm

In this way, the integers $l_{i+1},m_{i+1}$, and $k_{i+1}$ are defined without using the large
numbers $a_i$ and $b_i$. Still, if we want to include them, we must
take
\begin{equation}\label{mmmm}
    a_i=\frac{a_{i-1}m_{i+1}+d b_{i-1}l_{i+1}}{|k_i|},\qquad b_{\,i}=
    \frac{a_{i-1}l_{i+1}+b_{i-1}m_{i+1}}{|k_i|}.
\end{equation}
Observe that in both algorithms with $L$, (\ref{mmmm}) produces a product representation of the convergent
and the fundamental solution of the Pell equation as follows:

\begin{equation}\label{prod}
a_i+\sqrt{d} \,b_i=\big(a_1+\sqrt{d} \,b_1\big)\
\prod_{j=2}^{i}\ \bigg( \frac{m_{j+1}+\sqrt{d}\, l_{j+1}}{|k_j|}
\bigg)
\end{equation}
and (\ref{prod})
 is similar to what is called {\it "power product"} in
\cite{Len}.


\vskip .5cm

\noindent \textbf{Computation of the minimum in Step 2.} We present in detail some options in
the implementation of the algorithm depending strictly on different ways to  compute the minimum in Step 2.  First, observe that in (\ref{ver}),
 we are basically looking for a
couple of positive integers $l_{i+1}$ and $r_i$ that produce the following
minimization
\begin{equation}\label{verdura}
\min_{\substack{l=1,\dots ,L\\\hbox{\tiny all }r}}\ \big|\, r -\,l\,
\alpha_i \big|,
\end{equation}
where
\begin{equation*}\label{alpha}
\alpha_i=\frac{\sqrt{d} - M_i}{|k_i|}.
\end{equation*}
\vskip .2cm
\noindent Now, we recall
a basic result in the theory of simple continued fractions (see
p.340 in \cite{NZ}):  A pair of positive integers $a$ and $b$ is
called a {\it good approximation} to the positive irrational
number $\xi$ if
\begin{equation*}
|\, b\,\xi \, -\, a\,|\, =\ \min_{\substack{y=1,\dots
,b\\\hbox{\tiny all }x}}\ \big|\, y\, \xi - x \, \big|.
\end{equation*}
A classical result is that the pair $a$ and $b$ is a good approximation
of $\xi$ if and only if $a/b$ is a convergent of $\xi$. Using this result, we present two options for computing Step 2.

\

First, for a fixed $L$,
we can apply the method of continued fractions to $\alpha_i$ in order to find $l_{i+1}$ and $r_i$. If
$\tilde a_n/\tilde b_{\, n}$ is the $n$-th  convergent of $\alpha_i$ with $\tilde b_{\,
n}\leq L< \tilde b_{\, n+1}$, then, we take $l_{i+1}=\tilde b_{\, n}$ and $r_i=\tilde a_n$. These numbers minimize (\ref{verdura}) by
Theorem 7.13 in \cite{NZ}.  This variant of our algorithm is denoted as the {\bf Second Algorithm  with CF and L}.
We shall not present any example of
this implementation.

\

Secondly, we have the option to set (or adjust, if necessary) the number of steps $s$ used in applying the continued fractions algorithm to $\alpha_i$. Then, we set $L_i=l_{i+1}=b_s$ and $r_i=a_s$. Notably, $L$ is not fixed in this scenario. This particular variation of our algorithm is referred to as the {\bf Second Algorithm with CF and s}. Once more, we will refrain from providing an example of this implementation.

\

\begin{remark}
  Observe that the Second Algorithm with CF and L (resp. with CF and $s$), can also be used with different values of L (resp. $s$) as we did with the following LLL version.
\end{remark}

\vskip .5cm

\subsection{ Second Algorithm with LLL}\lbb{lL}

We present one of the most important parts of this work. This improvement in the implementation of the second
algorithm with L consists of replacing Step 2 with an approximation to $\al_i$  using the LLL-algorithm. The LLL-algorithm is the common notation for the \textit{Lenstra–Lenstra–Lovász (LLL) lattice basis reduction algorithm} defined in \cite{LLL}. It is a polynomial time lattice reduction algorithm. We only need to use it for lattices of rank 2, see Section 9 in \cite{Len2} for details. By Proposition 1.39 in \cite{LLL}, we have  the following result: given  rational numbers $\alpha$ and
$\varepsilon$, satisfying $0<\varepsilon<1$, the LLL-algorithm (for rank 2 lattices)
finds  integers $p$ and $q$ for which
\begin{equation*}
    |\, p-q\alpha\, |\, \leq \, \varepsilon \quad \hbox{ and }\quad
1\leq q \leq \frac{\text{\footnotesize $\sqrt{2}$}}{\varepsilon}.
\end{equation*}

\noindent Now, for a fixed positive integer $LLL$, we set
$\varepsilon=\frac{\text{\footnotesize $\sqrt{2}$}}{LLL}$ and choose
$\alpha$ as a   rational approximation of $\alpha_i$.
Therefore, the LLL-algorithm
finds  integers $p$ and $q$ such that
\begin{equation*}
    |\, p-q\alpha\, |\, \leq \, \frac{\text{\footnotesize $\sqrt{2}$}}{LLL} \quad \hbox{ and }\quad
1\leq q \leq LLL.
\end{equation*}
In our examples, we applied the "lattice" command in Maple 8 to the lattice
generated by $(1,0)$ and $(-\,\alpha,
\frac{\varepsilon^2}{\sqrt{2}})$, to produce the positive integers
$l_{i+1}=q$ and $r_i=p$ that we use to replace the Step 2. However, it's important to note that these values may not necessarily minimize  (\ref{verdura}).  For detailed information, refer to pages 139-140 in \cite{Len2}. Consequently, in some cases, we may not obtain the fundamental solution, as we will observe in certain examples.
This variant of our algorithm
is denoted as the {\bf Second Algorithm with LLL}.

\vskip .5cm

The following
table provides an example of the different algorithms applied to
$d=1234567890$, where the cases with $L$ correspond to the Second Algorithm with L. In this case, the fundamental solution $\epsilon$ has 1935
decimal digits (or $R_{\, d}^{^{_{10}}} = 1935$, using our notation for the regulator (with base 10) introduced after Proposition \ref{2.2.2}).
\begin{table}[H]
\caption{Number of steps when $d=1234567890$} \label{table:4} \
\begin{center}
\fontsize{9}{15}\selectfont
\begin{tabular}{ | c || c | c | c | c | c | c | c | c | c |}
\hline
 & CF & Chak. & $L=9$ & $L=100$ & $L=200$
 & $LLL=10^6$  & $LLL=10^{18}$ &$LLL=10^{20}$  &$LLL=10^{25}$ \\
\hline \hline  steps & 3772 & 2611 & 1302 & 768 & 690
& 304 & 105  & 95 & 76  \\
   \hline
\end{tabular}
\end{center}
\end{table}
Observe that in all the cases with $LLL$, the number of digits of the
fundamental solution is related to the value of $LLL$ and the number of steps. Specifically,  the fundamental solution
has $1935$ decimal digits, and with $LLL=10^{20}$ the number of
steps is $95$, hence we have $\frac{1935}{95}\simeq 20$. The same
holds for the other exponents with $LLL$: $\frac{1935}{105}\simeq
18$.
Again, the study of the convergence of the different algorithms
remains as an open problem, as well as the running time of them.

\vskip .5cm

In the following example, we take
$d=130940879$. In this case, the continued fraction algorithm needs
5259 steps, and the fundamental solution has 2727 decimal digits. Observe that $\lfloor\sqrt{R_{\, d}^{^{_{10}}}}\rfloor=52$. In this case, with the Second Algorithm with $LLL=10^{52}$, we needed 52 steps to obtain the fundamental solution.

Now, in the same example, \textbf{we use the Second Algorithm with LLL, but we apply it with
2 different speeds}, that is, we use a value of $LLL$ for a certain
fixed number of steps, and then we continue with another value of
$LLL$ for the remaining steps. More precisely,

\

\noindent $\bullet$  \ 27 steps with $LLL=10^{75}$ + more than 500 steps
with $LLL=10^{10}$: diverges.

\noindent $\bullet$ \  27 steps with $LLL=10^{75}$ + 128 steps with
$LLL=10^{5}$: 155 steps, we get solution $\epsilon$.

\noindent $\bullet$ \  35 steps with $LLL=10^{75}$ + more than 500 steps
with $LLL=10^{10}$: diverges.

\noindent $\bullet$ \  35 steps with $LLL=10^{75}$ + 523 steps with
$LLL=10^{5}$: 558 steps,  we get solution $\epsilon^2$.

\noindent $\bullet$ \  27 steps with $LLL=10^{100}$ + \ 3 \  steps with
$LLL=10^{5}$: 30 steps,  we get solution $\epsilon$.

\noindent $\bullet$ \  18 steps with $LLL=10^{150}$ +  510   steps with
$LLL=10^{5}$: 528 steps,  we get solution $\epsilon^2$.

\noindent $\bullet$ \  18 steps with $LLL=10^{150}$ + 3 steps with
$LLL=10^{6}$: 21 steps,  we get solution $\epsilon$.

\noindent $\bullet$ \  9 steps with $LLL=10^{250}$ + more than 500 steps
with $LLL=10^{10}$: diverges.

\noindent $\bullet$ \  9 steps with $LLL=10^{250}$ + 87 steps with
$LLL=10^{5}$: 96 steps,  we get solution $\epsilon$.

\noindent $\bullet$ \  11 steps with $LLL=10^{250}$ + 501 steps with
$LLL=10^{5}$: 512 steps,  we get solution $\epsilon^2$.

\noindent $\bullet$ \  9 steps with $LLL=10^{300}$ + \  4 steps with
$LLL=10^{5}$: 13 steps,  we get solution $\epsilon$.

\noindent $\bullet$ \  9 steps with $LLL=10^{300}$ + \  2 steps with
$LLL=10^{10}$: 11 steps,  we get solution $\epsilon$.

\noindent $\bullet$ \  9 steps with $LLL=10^{300}$ + \  1 step with
$LLL=10^{20}$: 10 steps,  we get solution $\epsilon$.

\noindent $\bullet$ \  6 steps with $LLL=10^{452}$ + 1 step with
$LLL=10^{10}$: 7 steps,  we get solution $\epsilon$.

\noindent $\bullet$  \  5 steps with $LLL=10^{542}$ + more than 500 steps
with $LLL=10^{10}$: diverges.

\noindent $\bullet$  \  5  steps with  $LLL=10^{542}$ + \ 3 steps with
$LLL=10^{5}$: {8} steps,  we get solution $\epsilon$.

\vskip .4cm

\noi Observe that in some cases the algorithm diverges, and in some other cases, we get the square of the fundamental solution. In most cases, we took the number of steps for the first value of $LLL$  to approximate to the floor of the regulator $R_{\, d}^{^{_{10}}} $, which is equal
to 2727. For example, in the last item in the previous list, we took 5 steps with $LLL=10^{542}$ since $5 \cdot 542=2710$, and if we add 3 steps with $LLL=10^{5}$, we have $5 \cdot 542 + 3 \cdot 5=2725\simeq 2727$.

\vskip .3cm

If we know a close lower bound of the
regulator $R_{\, d}^{^{_{10}}} $, then we can take a few steps with a big $LLL$ to
approximate it and then continue with a lower value of $LLL$ to
get the fundamental solution. The
approximation of the regulator that we need is not necessarily the
floor of the regulator, as in Section 6 in \cite{Len}. It is enough
to know a lower bound that should follow from the analysis of
convergence and the running time. What we essentially have done was the following: if $R_{\, d}^{^{_{10}}} \simeq k_d\, h_d +c_d$, where $k_d , h_d , c_d$ are positive integers, then we apply $k_d$ steps with $LLL=10^{h_d}$ and $c_d$ steps $LLL=10$, to obtain the fundamental solution.
Continuing with the same example,
and considering that $272\cdot 10=2720\simeq 2727$, we have the
following  case: if we take 10 steps with $LLL=10^{272}$,
 then the algorithm converges to $\epsilon$.

\vskip .3cm

If we accept a weak version of the folklore conjecture given in p.7 in \cite{Len}, then we may assume  the first inequality in
\begin{equation*}
  \frac{\sqrt{d}}{(\log d)^q} < R_{\, d}^{^{_{10}}}\cdot \log (10) < \sqrt{d} \, (\log (4d)+2),
\end{equation*}
for some positive integer $q$. In the examples presented on pages 345-348 in \cite{JW}, the regulators satisfy this inequality with $q=1$. Therefore, we may apply the previous idea to $\frac{\sqrt{d}}{\log (10)\cdot (\log d)^q}$, a value that is close to $R_{\, d}^{^{_{10}}}$, and take positive integers $k_d$ and $h_d$ such that  $\frac{\sqrt{d}}{\log (10)\cdot (\log d)^q}\simeq k_d\, h_d$.

The most important part of the running time is given by the implementation of the LLL-algorithm at each step. If it were necessary, for each $\alpha_i$,  we can take the appropriate $LLL_i$ to reduce the running time,  
by  using the remarks on pages 140 and 147 in \cite{Len2} about the implementation of the LLL-algorithm in the simple and special case of  lattices of rank 2.

\subsection{Computation of $M_i$ without involving $a_{i-1}$ and $b_{i-1}$}\label{5-M}

The results of this section hold for the First and Second Algorithms with L. Recall that we defined $M_i$ such that $0\leq M_i<|k_i|$ and
\begin{equation*}
    M_{i}\equiv -\frac{\ a_{i-1}}{\ b_{i-1}}  \ \ {\hbox{ (mod
    $|k_i|$)}}.
\end{equation*}

To obtain $M_i$ in a different way, we need some results.
 Using
the idea in (\ref{909}), together with (\ref{4.8}) and
(\ref{4.9}),  we have
\begin{equation}\label{nuuu}
b_{i-1}\,m_i
-\,a_{i-1} l_i= \hbox{sign}(k_{i-1})\, b_{i-2}\, k_i.
\end{equation}
Similarly, one can see that
\begin{equation}\label{nuuu2}
  a_{i-1} m_i - d\, b_{i-1} l_i=\hbox{sign}(k_{i-1})\, a_{i-2}\, k_i.
\end{equation}
Hence
\begin{equation}\label{m1}
m_i\, b_{i-1}\equiv l_i\, a_{i-1} \ \ (\hbox{mod }|k_i|),
\end{equation}
or
\begin{equation}\label{noo}
  m_i\equiv -M_i \, l_i \ \ \, (\hbox{mod }|k_i|).
\end{equation}
Observe that, from (\ref{4.51}) and (\ref{m1}), we obtain
\begin{equation}\label{mm2}
    l_i m_{i+1} \equiv -l_{i+1} m_i \quad\hbox{ (mod $|k_i|$) },
\end{equation}
and we recover (\ref{3.12}) by taking $1=L=l_i$ for all $i$ in
(\ref{mm2}), that is the {\it main equation} in the implementation of
the Chakravala algorithm, cf. (\ref{3.12}).

In certain steps within the examples, there might occur situations where $l_i$ and $k_i$ are not coprime. Hence  sometimes we cannot take  $
M_i\equiv -\frac{m_{i}}{l_{i}}\, \ \hbox{ (mod $|k_i|$) }$ in (\ref{noo}).

Recall from  (\ref{4.53}), that we have
\begin{equation} \label{nuevo}
    m_{i+1}=M_i \,l_{i+1} + r_i \ |k_i|,
\end{equation}

\vskip .1cm

\noindent for some integer $r_i$. Hence, by  definition we have that $r_{i-1}=\frac{m_i -M_{i-1} l_i }{|k_{i-1}|}$, and now we define
\begin{equation*}
  s_{i-1}:=\frac{d \, l_i -M_{i-1} m_i }{|k_{i-1}|}.
\end{equation*}

\vskip .1cm

\begin{proposition} \label{P9} The following properties hold:

\vskip .1cm

   (a) $s_{i-1}$ is an integer.

\vskip .1cm

 (b) $a_{i-1} r_{i-1} \equiv b_{i-1} s_{i-1} \, \ \hbox{ (mod $|k_i|$) }$.

\vskip .1cm

  (c) $(l_i, r_{i-1})=1$.

 \vskip .1cm

  (d) $M_i\equiv M_{i-1}  \, \ (\hbox{mod $(l_i, k_i)$})$.
\end{proposition}

\begin{proof}
    (a) By (\ref{nuevo}), we have that $M_{i-1} m_i\equiv M_{i-1}^2 l_i \ \hbox{ (mod $|k_{i-1}|$) }$. Using that $a_{i-2}^2-d\, b_{i-2}^2=k_{i-1}$, we obtain that $M_{i-1}^2\equiv \left(\frac{a_{i-2}}{b_{i-1}}\right)^2\equiv d \hbox{ (mod $|k_{i-1}|$)}$. Therefore, $M_{i-1} m_i\equiv d\, l_i \hbox{ (mod $|k_{i-1}|$)}$,  finishing the proof of (a).

    (b) Using (\ref{nuuu}) and (\ref{nuuu2}), observe that
   \begin{align*}
     a_{i-1} r_{i-1} - b_{i-1} s_{i-1} & = a_{i-1} \left( \frac{m_i -M_{i-1} l_i}{|k_{i-1}|}\right) -  b_{i-1} \left(\frac{  d \, l_i -M_{i-1} m_i}{|k_{i-1}|}\right) \\
       & =\frac{1}{|k_{i-1}|}\Big(M_{i-1} \big( b_{i-1} m_i - a_{i-1} l_i \big) + \big(a_{i-1} m_i - d\, b_{i-1} l_i\big) \Big)\\
       & = \frac{k_i}{k_{i-1}}\big(M_{i-1} b_{i-2} + a_{i-2}\big),
   \end{align*}
   but $(M_{i-1} b_{i-2} + a_{i-2}\big)$ is a multiple of $k_{i-1}$ by the definition of $M_{i-1}$, finishing the proof of (b).

   (c) Suppose that $n_i=(l_i,r_{i-1})>1$. Since $m_i=M_{i-1} l_i + r_{i-1} |k_{i-1}|$, then $n_i$ divides $m_i$. Define $\widetilde m_i=m_i/n_i$, $\widetilde l_i=l_i/n_i$ and $\tilde r_{i-1}=r_{i-1}/n_i$. Then $\widetilde m_i$ and $\widetilde l_i$ satisfy  $\widetilde m_i=M_{i-1} \widetilde l_i + \tilde r_{i-1} |k_{i-1}|$, i.e., $k_{i-1} $ divides $a_{i-2}\, \widetilde l_i + b_{i-2}\, \widetilde m_i$, and
   \begin{equation*}
     |\,m_i - \sqrt d \, l_i |=n_i \, |\,\widetilde m_i - \sqrt d \,\widetilde l_i |,
   \end{equation*}
   which  contradicts the minimality of $|\, m_i - \sqrt d \, l_i |$ in the Second Algorithm with L. Similarly, a contradiction is obtained for the First Algorithm with L, concluding the proof of (c).

   (d) Finally, observe that $(l_i,k_i)|\,m_i$ by (\ref{noo}). Hence, using that $m_i=M_{i-1} l_i+r_{i-1} |k_{i-1}|$, we obtain that $(l_i,k_i)|(r_{i-1} |k_{i-1}|)$. Therefore, using part (c), we have that $(l_i,k_i)|k_{i-1}$ and $(l_i,k_i)=(l_i,k_i,k_{i-1})$.  Now, by (\ref{000}), we have  $a_{i-1}b_{i-2}\equiv b_{i-1} a_{i-2} \ (\hbox{mod } l_i)$. In particular, $a_{i-1}b_{i-2}\equiv b_{i-1} a_{i-2}\  (\hbox{mod } (l_i,k_i,k_{i-1}))$. Hence we can take the quotients to get
     \begin{equation*}
       \frac{a_{i-1}}{b_{i-1}}\equiv \frac{a_{i-2}}{b_{i-2}}\  (\hbox{mod } (l_i,k_i)).
     \end{equation*}
From this, (d) is obtained.
\end{proof}

\

\noindent
\textbf{Summary of the computation of $M_i$:}

\vskip .2cm

\noindent $\ast$ \, If $(l_i,k_i)=1$, then by (\ref{noo}), we take
\begin{equation*}
  M_i\equiv - \frac{m_i}{l_i} \ \, (\hbox{mod }|k_i|).
\end{equation*}
Suppose that  $(l_i,k_i)>1$, then
\begin{equation}\label{11111}
  M_i \equiv - \frac{\widetilde m_i}{{\widetilde l_i}} \ \, (\hbox{mod }|\widetilde k_i|),
\end{equation}
where $\widetilde m_i= \frac{m_i}{(l_i,k_i)}, \widetilde l_i = \frac{l_i}{(l_i,k_i)}$ and $\widetilde k_i= \frac{k_i}{(l_i,k_i)}$. Observe that $\widetilde m_i$ is an integer since $(l_i,k_i)|\,m_i$ by (\ref{noo}).

\vskip .2cm

\noindent $\ast$ \, If $(l_i,k_i)>1$ and $\left( (l_i,k_i), \frac{k_i}{(l_i,k_i)} \right)=1$, then using (\ref{11111}) together with
$M_i\equiv M_{i-1}  \, \ (\hbox{mod $(l_i, k_i)$})$, and the Chinese remainder theorem, we can get $M_i$.

\vskip .2cm

\noindent From now on, we suppose that  $(l_i,k_i)>1$ and $\left( (l_i,k_i), \frac{k_i}{(l_i,k_i)} \right)>1$

\vskip .2cm

\noindent $\ast$ \, If $(r_{i-1},k_i)=1$, then using Proposition \ref{P9} (b), we take
\begin{equation*}
  M_i\equiv - \frac{s_{i-1}}{r_{i-1}} \ \, (\hbox{mod }|k_i|).
\end{equation*}

\vskip .2cm

\noindent $\ast$ \, If $(r_{i-1},k_i)>1$, then we have
\begin{equation}\label{22222}
  M_i\equiv - \frac{\widehat s_{i-1}}{\widehat r_{i-1}} \ \, (\hbox{mod }|\widehat k_i|),
\end{equation}
where $\widehat r_{i-1}= \frac{r_{i-1}}{(r_{i-1},k_i)},   \widehat s_{i-1}= \frac{s_{i-1}}{(r_{i-1},k_i)}$ and $\widehat k_{i}= \frac{k_{i}}{(r_{i-1},k_i)}$. Using Proposition \ref{P9}, that is $(l_i,r_{i-1})=1$, we have that $k_i$ divides $\widehat k_{i} \cdot \widetilde k_i$. Then we can take $\widehat K_{i}$ (resp. $\widetilde K_i$) a divisor of $\widehat k_{i}$ (resp. $\widetilde k_i$) such that
$(\widehat K_{i},\widetilde K_i)=1$ and $k_i=\widehat K_{i} \cdot \widetilde K_i$. Finally, using these divisors in
 (\ref{11111}) and (\ref{22222}), we can apply  the Chinese remainder theorem to obtain the value of $M_i$.

\vskip .4cm

This concludes the computation of $M_i$ without using $a_{i-1}$ and $b_{i-1}$. There might be a simpler approach to achieve this. Additionally, it remains uncertain whether an implementation exists where one can ensure, at each step, that $k_i$ divides $a_{i-1} l_{i+1} + b_{i-1} m_{i+1}$ without relying on the values of
$M_i$.

\vskip .2cm

\section{Generalization of the Continued Fraction Algorithm with an Additional Variable}\lbb{7ma}

In this section, we introduce a generalization of the continued fraction algorithm, drawing inspiration from the ideas of the First Algorithm with L. Note that in the continued fraction algorithm, $q_n$ in
(\ref{vvver}) corresponds to selecting $P_{n+1}$  as the positive integer
congruent to $-P_n$  module $Q_n$ (where all $Q_n$ are positive) satisfying
\begin{equation*}
\sqrt{d} -P_{n+1} \, >0\ \hbox{  and\  minimizing this value},
\end{equation*}
or equivalently
\begin{equation*}
{d} -P_{n+1}^2 \, >0\ \hbox{  and\  minimizing this value}.
\end{equation*}

\vskip .2cm

\noindent With this motivation, we can attempt to define a kind of continued fraction algorithm with L by replacing the minimization conditions in the First and Second Algorithms with L with:
\begin{equation}\label{3CF}
l_{i+1} \sqrt{d} -m_{i+1}\, >0\ \hbox{  and\  minimizing this value},
\end{equation}
or
\begin{equation}\label{4CF}
 l_{i+1}^{\,2}\, {d} -m_{i+1}^2\, >0\ \hbox{  and\  minimizing this value}.
\end{equation}

\vskip .3cm

More precisely, we will use (\ref{4CF}) to define the \textbf{Continued Fraction Algorithm  with L} as follows: given a  non-square positive integer $d$ and a
positive integer $L$, the algorithm produces the sequences of
integers $a_i, b_i, k_i, m_i, l_i$ by the following recipe:
we start with  the first step in Chakravala algorithm as in Remark \ref{inicio}:
\begin{align*}\label{inicio9}
& \mathrm{if \ \ }(\lfloor\sqrt{d}\rfloor+1)^2-d<d-\lfloor\sqrt{d}\rfloor^2
,\mathrm{\ \ then\ \ }  a_1=\lfloor\sqrt{d}\rfloor+1,\,
b_1=1,\,k_2=a_1^2-d\, b_1^2, \\
& \mathrm{else:\ \  } a_1=\lfloor\sqrt{d}\rfloor, \, b_1=1,\,k_2=a_1^2-d\, b_1^2.\nonumber
\end{align*}
Now, given integers $a_{i-1},b_{i-1}$ and $k_i$ where
$(a_{i-1},b_{i-1})=1$ such that
\begin{equation*}
    a_{i-1}^2-d \, b_{i-1}^2=k_i,
\end{equation*}
we choose $m_{i+1}$ and $l_{i+1}$ to be a positive integers for
which $a_{i-1}l_{i+1}+b_{i-1}m_{i+1}$ is divisible by $k_i$ and
\begin{equation}\label{9CF}
 l_{i+1}^{\,2} \, {d} -m_{i+1}^2\, >0\ \hbox{  and\  minimizing this value for
$1\leq l_{i+1}\leq L$}.
\end{equation}

\vskip .2cm

\noindent
Then we take
\begin{equation*}
k_{i+1}=\frac{m_{i+1}^2-d\,l_{i+1}^2}{k_i}
\end{equation*}
and
\begin{equation*}
    a_i=\frac{a_{i-1}m_{i+1}+d \, b_{i-1}l_{i+1}}{|k_i|},\qquad b_i=
    \frac{a_{i-1}l_{i+1}+b_{i-1}m_{i+1}}{|k_i|}.
\end{equation*}
As before, using (\ref{4.33}), (\ref{4.4}) and $(a_{i-1},b_{i-1})=1$, we obtain that $k_{i+1}$ and $a_i$
are integers. By  a direct computation, we have
\begin{equation*}
    a_i^2-d \, b_i^2=k_{i+1},
\end{equation*}
and  to complete the recursive definition, we need to
prove that $(a_i,b_i)=1$. This is immediate by the ideas in the proof of  Proposition {\ref{p.1}}. The method
terminates when $k_n=1$ for some $n$. \textit{In the examples, we always obtained the fundamental solution,} except for extremely large values of $L$ relative to $d$. It needs more steps than the First Algorithm with L. As before, we identify the study of convergence and  its computational complexity  as open problems.

Finally, if we replace (\ref{9CF}) by (\ref{3CF}) in the previous algorithm, the algorithm does not work.

\vskip .3cm

\section{Some   Formulas and Final Remarks}\lbb{first}

\

In this section, we present some formulas analogous to well-known formulas for continued fractions written in Section 2. These formulas apply to the Second Algorithm with L. By definition and a simple computation, we have
 \begin{align*}
  a_i^2 - d \, b_i^2 & =k_{i+1} \nonumber \\
  a_i \, a_{i-1} - d\, b_i\,  b_{i-1} & = \mathrm{sign}(k_i) \ m_{i+1}
 \end{align*}
 which are the analogs to equations (3.13) in \cite{JW}. By the proof of Proposition \ref{p.1}, we have
  \begin{equation}\label{000}
    a_i\, b_{i-1} - b_i \,a_{i-1} = - \mathrm{sign}(k_i) \ l_{i+1}
  \end{equation}
that corresponds to (\ref{4}). The following formulas were proved in examples. These are the version with $L$ of (\ref{3}):
\begin{align*}
  l_{i+1} a_{i+1}  & = q_{i+1} a_i - \mathrm{sign}(k_i)\  \mathrm{sign}(k_{i+1})\ l_{i+2}\, a_{i-1} \\
  l_{i+1} b_{i+1}  & = q_{i+1} b_i - \mathrm{sign}(k_i)\ \mathrm{sign}(k_{i+1})\  l_{i+2}\,  b_{i-1},
\end{align*}

\noindent where $q_i$ is defined by $l_i \, m_{i+1} + m_i\,  l_{i+1} = q_i \, |k_i|$, which is a positive integer by (\ref{mm2}). These are the version with $L$ of (\ref{19}):
\begin{align*}
  l_{i+1}\, d\, b_i & = m_{i+1}\, a_i - \mathrm{sign}(k_i)\ k_{i+1}\, a_{i-1}  \\
  l_{i+1}\, a_i & = m_{i+1}\, b_i - \mathrm{sign}(k_i)\ k_{i+1}\, b_{i-1}.
\end{align*}

Now, we present a formula analogous to (\ref{PPHI}), that is, a kind of  interpretation of this algorithm in terms
of a generalized continued fraction expansion of $\sqrt{d}$. We define (cf. (\ref{7000}))
\begin{equation*}
   \phi_i=\frac{(m_i + l_i \, \sqrt{d})}{|k_i|}\, l_{i+1}.
\end{equation*}
By the standard computation, one can prove that
\begin{equation*}
   \phi_{i+1}=\frac{s\, l_i\,l_{i+2}}{\phi_i -q_i}
\end{equation*}
where $q_i$ was defined above, and $s=-\mathrm{sign}(k_i)\ \mathrm{sign}(k_{i+1})$.

\vskip .8cm

\subsection*{Acknowledgements} Dedicated to my parents, H\'ector and Sheila, with heartfelt gratitude.


\bibliographystyle{amsalpha}


\end{document}